\documentclass[12pt]{amsart}
\usepackage{amssymb}
\usepackage[all]{xy}
\usepackage{mathrsfs}
\newtheorem{theorem}{Theorem}[]
\newtheorem{lemma}[theorem]{Lemma}
\newtheorem{problem}[theorem]{Problem}
\newtheorem{corollary}[theorem]{Corollary}
\def\r{\mathrm}
\def\to{\rightarrow}
\def\sl{\langle}
\def\sr{\rangle}
\begin{document}
%\baselineskip15pt
%%%%%%%%%%%%%%%%%%%%%%%%%%%%%%%%%%%%%%%%%%%%%%%%%%%%%%%%%%%%%%%%%%%%%%%%%%%%%%%%%%%%%%%%%%%%%%%%%%%%%%%%%%%%%%%%%%%%%%%%%%
\title[Decomposition of functions in the orthogonality equation]
{Decomposition of functions between Banach spaces in the orthogonality equation}
%%%%%%%%%%%%%%%%%%%%%%%%%%%%%%%%%%%%%%%%%%%%%%%%%%%%%%%%%%%%%%%%%%%%%%%%%%%%%%%%%%%%%%%%%%%%%%%%%%%%%%%%%%%%%%%%%%%%%%%%%%
\author[M.M. Sadr]{Maysam Maysami Sadr}
\address{Department of Mathematics\\
Institute for Advanced Studies in Basic Sciences\\
P.O. Box 45195-1159, Zanjan 45137-66731, Iran}
\email{sadr@iasbs.ac.ir}
%%%%%%%%%%%%%%%%%%%%%%%%%%%%%%%%%%%%%%%%%%%%%%%%%%%%%%%%%%%%%%%%%%%%%%%%%%%%%%%%%%%%%%%%%%%%%%%%%%%%%%%%%%%%%%%%%%%%%%%%%%
%%%%%%%%%%%%%%%%%%%%%%%%%%%%%%%%%%%%%%%%%%%%%%%%%%%%%%%%%%%%%%%%%%%%%%%%%%%%%%%%%%%%%%%%%%%%%%%%%%%%%%%%%%%%%%%%%%%%%%%%%%
\begin{abstract}
Let $E,F$ be Banach spaces. In the case that $F$ is reflexive we give a description for the solutions $(f,g)$
of the Banach-orthogonality equation
$$\sl f(x),g(\alpha)\sr=\sl x,\alpha\sr\hspace{10mm}\forall x\in E,\forall \alpha\in E^*,$$
where $f:E\to F,g:E^*\to F^*$ are two maps. Our result generalizes the recent result of {\L}ukasik and W\'{o}jcik in the case that
$E$ and $F$ are Hilbert spaces.
\end{abstract}
%%%%%%%%%%%%%%%%%%%%%%%%%%%%%%%%%%%%%%%%%%%%%%%%%%%%%%%%%%%%%%%%%%%%%%%%%%%%%%%%%%%%%%%%%%%%%%%%%%%%%%%%%%%%%%%%%%%%%%%%%%
%%%%%%%%%%%%%%%%%%%%%%%%%%%%%%%%%%%%%%%%%%%%%%%%%%%%%%%%%%%%%%%%%%%%%%%%%%%%%%%%%%%%%%%%%%%%%%%%%%%%%%%%%%%%%%%%%%%%%%%%%%
\maketitle
%%%%%%%%%%%%%%%%%%%%%%%%%%%%%%%%%%%%%%%%%%%%%%%%%%%%%%%%%%%%%%%%%%%%%%%%%%%%%%%%%%%%%%%%%%%%%%%%%%%%%%%%%%%%%%%%%%%%%%%%%%
%%%%%%%%%%%%%%%%%%%%%%%%%%%%%%%%%%%%%%%%%%%%%%%%%%%%%%%%%%%%%%%%%%%%%%%%%%%%%%%%%%%%%%%%%%%%%%%%%%%%%%%%%%%%%%%%%%%%%%%%%%
%%%%%%%%%%%%%%%%%%%%%%%%%%%%%%%%%%%%%%%%%%%%%%%%%%%%%%%%%%%%%%%%%%%%%%%%%%%%%%%%%%%%%%%%%%%%%%%%%%%%%%%%%%%%%%%%%%%%%%%%%%
%%%%%%%%%%%%%%%%%%%%%%%%%%%%%%%%%%%%%%%%%%%%%%%%%%%%%%%%%%%%%%%%%%%%%%%%%%%%%%%%%%%%%%%%%%%%%%%%%%%%%%%%%%%%%%%%%%%%%%%%%%
\section{Introduction}
The topological dual of a Banach space $E$ is denoted by $E^*$, and $\sl\cdot,\cdot\sr$ denotes the usual pairing between elements of $E$ and $E^*$.
Let $E,F$ be Banach spaces and $f:E\to F$, $g:E^*\to F^*$ be two maps. Consider the following functional equation,
which we call \emph{Banach-orthogonality equation}, for unknown maps $f,g$:
\begin{equation}\label{e1}
\sl f(x),g(\alpha)\sr=\sl x,\alpha\sr\hspace{10mm}\forall x\in E,\forall \alpha\in E^*.
\end{equation}
The \emph{generalized orthogonality equation} introduced in \cite{LukasikWojcik1} and \cite{Chmielinski1}
is a special case of (\ref{e1}) where $E$ and $F$ are Hilbert spaces, $E^*$ and $F^*$
are identical with $E$ and $F$ via the canonical (conjugate) linear isometric isomorphisms induced by the inner products, and $\sl\cdot,\cdot\sr$
denotes the inner products. {\L}ukasik and W\'{o}jcik \cite{LukasikWojcik1} have recently characterized the solutions $(f,g)$ of the generalized
orthogonality equation. (See also \cite{Lukasik1} in the case that $E,F$ are pre-Hilbert spaces.) In this paper we give a characterization
of the solutions $(f,g)$ of (\ref{e1}) in the case that $F$ is a reflexive Banach space, via decompositions of $f,g$ by
invertible linear and arbitrary nonlinear parts. Then the result of \cite{LukasikWojcik1} is seen as a special case of our main result.
In the remainder of this section we consider two preliminary lemmas and in the next section we give the main result.
The proof of the following lemma is similar to the proofs of Lemmas 1 and 2 of \cite{LukasikWojcik1}. For the sake of completeness we
add the proof.
\begin{lemma}\label{l1}
Let $C,D$ be Banach spaces and $S:C\to D,T:C^*\to D^*$ be two maps satisfying
$$\sl S(x),T(\alpha)\sr=\sl x,\alpha\sr\hspace{10mm}\forall x\in C,\forall \alpha\in C^*.$$
\begin{enumerate}
\item[(i)] If $\overline{\r{Lin}S(C)}=D$ then $T$ is a bounded linear operator.
\item[(ii)] If $\overline{\r{Lin}T(C^*)}=D^*$ then $S$ is a bounded linear operator.
\end{enumerate}
\end{lemma}
\begin{proof}
(i) For every $\alpha,\alpha',x$ we have
$$\sl S(x),T(\alpha+\alpha')\sr=\sl x,\alpha+\alpha'\sr=\sl S(x),T(\alpha)+T(\alpha')\sr,$$
and thus $T(\alpha+\alpha')|_{S(C)}=(T(\alpha)+T(\alpha'))|_{S(C)}$. Since $T(\alpha+\alpha'),T(\alpha)+T(\alpha')$
are bounded linear functionals and $\overline{\r{Lin}S(C)}=D$ we have $T(\alpha+\alpha')=T(\alpha)+T(\alpha')$. Similarly, for every scalar
$r$ we have $T(r\alpha)=rT(\alpha)$. So it was proved that $T$ is linear. Suppose that $\alpha_n\to\alpha$ and $T(\alpha_n)\to\beta$.
We have $\sl S(x),T(\alpha_n)\sr=\sl x,\alpha_n\sr\to\sl x,\alpha\sr=\sl S(x),T(\alpha)\sr$. On the other hand,
$\sl S(x),T(\alpha_n)\sr\to\sl S(x),\beta\sr$. Thus,  $\sl S(x),T(\alpha)\sr=\sl S(x),\beta\sr$ and hence $T(\alpha)|_{S(C)}=\beta|_{S(C)}$.
It follows that $T(\alpha)=\beta$. Now since $C^*,D^*$ are Banach spaces Closed Graph Theorem implies that $T$ is continuous.
The proof of (ii) is similar.
\end{proof}
Let $C$ be a Banach space. For subsets $U\subseteq C$ and $V\subset C^*$ we let
$$U^\bot:=\{\alpha\in C^*: \sl x,\alpha\sr=0, \forall x\in U\},\hspace{5mm} V^\bot:=\{x\in C: \sl x,\alpha\sr=0, \forall \alpha\in V\}.$$
It is easily checked that $U^\bot$ and $V^\bot$ are closed linear subspaces. Let $K$ be a closed linear subspace of $C$.
Then the quotient vector space $C/K$ is a Banach space with the quotient norm defined by $\|x+K\|:=\inf_{x'\in K}\|x+x'\|$.
It is clear that $(C/K)^*$ embeds canonically in $C^*$ via the map $(C/K)^*\xrightarrow{I}C^*$ defined by $\alpha\mapsto \alpha P$ ($\alpha\in(C/K)^*$)
where $C\xrightarrow{P}C/K$ denotes the canonical projection. We also call the map $C^*\xrightarrow{R}K^*$,
defined by $\alpha\mapsto\alpha|_K$, the canonical restriction.
\begin{lemma}\label{l2}
Let $C$ be a reflexive Banach space and $W$ be a closed linear subspace of $C^*$. Then $W=W^{\bot\bot}$.
It follows that $(C/W^\bot)^*$ coincides with $W$  via the embedding $I$ defined as above with $K:=W^\bot$.
\end{lemma}
\begin{proof}
It is clear that $W\subseteq W^{\bot\bot}$. We must show the reverse inclusion. Assume, to reach a contradiction, that there exists an $\alpha$
in $W^{\bot\bot}\setminus W$. By geometric version of the Hahn-Banach Theorem and reflexivity of $C$ there is a $x\in C$ with $\sl x,\alpha\sr\ne0$
and $\sl x,\alpha'\sr=0$ for every $\alpha'\in W$. This contradicts $\alpha\in W^{\bot\bot}$. Thus $W=W^{\bot\bot}$. For the proof of the second
part, let $P$ be as above with $K:=W^\bot$. Let $\beta\in(C/W^\bot)^*$. We must show that $\beta P\in W$.
But it is obvious because $\beta P\in W^{\bot\bot}=W$.
\end{proof}
%%%%%%%%%%%%%%%%%%%%%%%%%%%%%%%%%%%%%%%%%%%%%%%%%%%%%%%%%%%%%%%%%%%%%%%%%%%%%%%%%%%%%%%%%%%%%%%%%%%%%%%%%%%%%%%%%%%%%%%%%%
%%%%%%%%%%%%%%%%%%%%%%%%%%%%%%%%%%%%%%%%%%%%%%%%%%%%%%%%%%%%%%%%%%%%%%%%%%%%%%%%%%%%%%%%%%%%%%%%%%%%%%%%%%%%%%%%%%%%%%%%%%
%%%%%%%%%%%%%%%%%%%%%%%%%%%%%%%%%%%%%%%%%%%%%%%%%%%%%%%%%%%%%%%%%%%%%%%%%%%%%%%%%%%%%%%%%%%%%%%%%%%%%%%%%%%%%%%%%%%%%%%%%%
%%%%%%%%%%%%%%%%%%%%%%%%%%%%%%%%%%%%%%%%%%%%%%%%%%%%%%%%%%%%%%%%%%%%%%%%%%%%%%%%%%%%%%%%%%%%%%%%%%%%%%%%%%%%%%%%%%%%%%%%%%
\section{The Main Result}
Let $S:C\to D$ be a bounded linear operator. Then its adjoint $S^*:D^*\to C^*$ is a bounded linear operator defined by
$\beta\mapsto \beta S$ for every $\beta\in D^*$. Thus we have the following identity.
$$\sl Sx,\beta\sr=\sl x,S^*\beta\sr\hspace{10mm}\forall x\in C,\forall \beta\in D^*.$$
It is easily checked that $S$ is invertible if and only if $S^*$ is invertible. Also if $S^*$ is surjective then $S$ is injective.
Now we are ready to state our main result.
\begin{theorem}\label{t1}
Let $E,F$ be Banach spaces such that $F$ is reflexive. Let $f:E\to F$, $g:E^*\to F^*$ be two maps.
Then the pair $(f,g)$ satisfies (\ref{e1}) if and only if there are
\begin{enumerate}
\item[(i)] closed linear subspaces $M\subseteq L\subseteq F$,
\item[(ii)] an invertible bounded linear operator $A:E\to L/M\subseteq F/M$,
\item[(iii)] a (not necessarily linear) right inverse $\varphi$ of the canonical projection $L\xrightarrow{P}L/M$,
i.e. $P\varphi=\r{id}_{L/M}$, and
\item[(iv)] a (not necessarily linear) right inverse $\psi$ of the canonical restriction $F^*\xrightarrow{R}L^*$, i.e. $R\psi=\r{id}_{L^*}$,
\end{enumerate}
such that
$$f=\varphi A,\hspace{10mm}g=\psi I(A^*)^{-1},$$
where $I$ denotes the canonical injection $(L/M)^*\xrightarrow{I}L^*$.
\end{theorem}
\begin{proof}
The ``if'' part of the theorem is easily checked. We only prove the ``only if'' part.

Let $L:=\overline{\r{Lin}f(E)}\subseteq F$. Let $F^*\xrightarrow{R}L^*$ denote the canonical restriction and let $Q_0:=Rg:E^*\to L^*$.
We have $\sl f(x),Q_0(\alpha)\sr=\sl x,\alpha\sr$, $\forall x\in E,\forall \alpha\in E^*$.
Applying Lemma \ref{l1}(i) with $C:=E,D:=L,S:=f,T:=Q_0$, we find that $Q_0$ is a bounded linear operator.
Note that there is a right inverse $\psi$ of $R$ such that $g=\psi Q_0$.

Let $M:=Q_0(E^*)^\bot=\overline{Q_0(E^*)}^\bot\subseteq L$. Thus $Q_0(\alpha)|_M=0$ for every $\alpha\in E^*$. This shows that
$Q_0$ actually takes the elements of $E^*$ to the space $(L/M)^*$ that is considered as a subspace of $L^*$ via the canonical injection
$(L/M)^*\xrightarrow{I}L^*$. We let $\hat{Q}_0$ denote the same operator $Q_0$ but with the new codomain $(L/M)^*$.
Thus we have $Q_0=I\hat{Q}_0$ and also $g=\psi I\hat{Q}_0$.

Since $F$ is reflexive, $L$ is reflexive and thus it follows from Lemma \ref{l2} that
$\overline{\hat{Q}_0(E^*)}=\overline{Q_0(E^*)}=(L/M)^*\subseteq L^*$. So $\hat{Q}_0$ is a bounded
linear operator from $E^*$ to $(L/M)^*$ with dense range.

Let $Q_1:=Pf$ where $L\xrightarrow{P}L/M$ denotes the canonical projection. Thus there is a right inverse $\varphi$ of $P$ such that $f=\varphi Q_1$.
We also have $\sl Q_1(x),\hat{Q}_0(\alpha)\sr=\sl x,\alpha\sr$, $\forall x\in E,\forall \alpha\in E^*$.
Applying Lemma \ref{l1}(ii) with $C:=E,D:=L/M,S:=Q_1,T:=\hat{Q}_0$, we find that $Q_1$ is a bounded linear operator.

Let $x$ be an arbitrary element of $E$. Let $\alpha\in E^*$ be such that $\|\alpha\|=1$ and $\sl x,\alpha\sr=\|x\|$. (The existence of such a functional
is a consequence of the Hahn-Banach theorem.) We have
$$\|x\|=|\sl x,\alpha\sr|=|\sl Q_1(x),\hat{Q}_0(\alpha)\sr|\leq\|\hat{Q}_0\|\|Q_1(x)\|.$$
This implies that for every $x\in E$, $\frac{\|x\|}{\|\hat{Q}_0\|}\leq \|Q_1(x)\|$. It follows that the range of $Q_1$ is a closed linear subspace of $L/M$.
On the other hand, since $\overline{\r{Lin}f(E)}=L$, we have $\overline{Q_1(E)}=L/M$. Thus $Q_1$ is a surjective operator.

For every $x\in E$ and $\alpha\in E^*$ we have
$$\sl x,\alpha\sr=\sl Q_1(x),\hat{Q}_0(\alpha)\sr=\sl x,Q_1^*\hat{Q}_0(\alpha)\sr.$$
This implies that $Q_1^*\hat{Q}_0=\r{id}_{E^*}$. Thus $Q_1^*$ is surjective and $Q_1$ is injective. It follows that $Q_1$ is invertible and
$\hat{Q}_0=(Q_1^*)^{-1}$. Now, we let $A:=Q_1$. The proof is complete.
\end{proof}
The following corollary of Theorem \ref{t1} is the main result of \cite{LukasikWojcik1}.
\begin{corollary}
Let $E$,$F$ be Hilbert spaces and $f,g:E\to F$ be two maps. Suppose that $f$,$g$ satisfy the generalized orthogonality equation:
$$\sl f(x),g(y)\sr=\sl x,y\sr,\hspace{10mm}\forall x,y\in E,$$
where $\sl\cdot,\cdot\sr$ denotes the inner products. Then there exist
\begin{enumerate}
\item[(i)] pairwise orthogonal closed linear subspaces $F_1,F_2,F_3$ of $F$ such that
$$F=F_1\oplus F_2\oplus F_3,$$
\item[(ii)] an invertible bounded linear operator $B:E\to F_1$,
\item[(iii)] a (not necessarily linear) map $\mu:E\to F_2$, and
\item[(iv)] a (not necessarily linear) map $\nu:E\to F_3$,
\end{enumerate}
such that
$$f=B+\mu\hspace{10mm}g=(B^*)^{-1}+\nu.$$
\end{corollary}
\begin{proof}
We identify $E^*$ and $F^*$ with $E$ and $F$ via the canonical (conjugate) linear isomorphisms induced by inner products.
Let $L,M,A,\varphi,\psi$ be as in the formulation of Theorem \ref{t1}. Let $F_3$ be the orthogonal complement of $L$ in $F$, let $F_2:=M$, and let
$F_1$ be the orthogonal complement of $M$ in $L$. Thus $L/M$ is canonically identified with $F_1$ and $L=F_1\oplus F_2$.
We let $B$ be the same operator $A$ but with codomain $F_1\cong L/M$.
Let $F_1\xrightarrow{\hat{\varphi}}F_2\oplus F_1$ denote the same map $L/M\xrightarrow{\varphi}L$. Thus there is a unique map
$F_1\xrightarrow{\tilde{\varphi}}F_2$ such that $\hat{\varphi}=\tilde{\varphi}+\r{id}_{F_1}$. We let $\mu:=\tilde{\varphi}B$.
Let $L\xrightarrow{\hat{\psi}}L\oplus F_3$ denote the same map $L^*\xrightarrow{\psi}F^*$. Thus there is a unique map
$L\xrightarrow{\tilde{\psi}}F_3$ such that $\hat{\psi}=\r{id}_L+\tilde{\psi}$. We let $\nu:=\tilde{\psi}(B^*)^{-1}$. This completes the proof.
\end{proof}
\begin{problem}
Characterize the solutions of (\ref{e1}) in the case that $F$ is not reflexive.
\end{problem}
A natural generalization of the problem considered in this paper is as follows.
\begin{problem}
Let $(E_1,E_2)$ be a pair of Banach spaces together with a pairing $\sl\cdot,\cdot\sr$ that is a (non degenerate)
bilinear functional on $E_1\times E_2$.
Let $(F_1,F_2)$ be another pair of Banach spaces with the pairing $[\cdot,\cdot]$. Characterize the solutions
$(E_1\xrightarrow{f_1}F_1,E_2\xrightarrow{f_2}F_2)$ of the following functional equation.
$$[f_1(x_1),f_2(x_2)]=\sl x_1,x_2\sr\hspace{10mm}\forall x_1\in E_1,\forall x_2\in E_2.$$
\end{problem}
Another possible extension is related to the theory of Hilbert C*-modules \cite{Lance1}:
\begin{problem}
Characterize the solutions of equation (\ref{e1}) in the case that $E$ and $F$ are Hilbert C*-modules over a C*-algebra $\mathcal{A}$
and $\sl\cdot,\cdot\sr$ denotes the $\mathcal{A}$-valued inner products.
\end{problem}
%%%%%%%%%%%%%%%%%%%%%%%%%%%%%%%%%%%%%%%%%%%%%%%%%%%%%%%%%%%%%%%%%%%%%%%%%%%%%%%%%%%%%%%%%%%%%%%%%%%%%%%%%%%%%%%%%%%%%%%%%%
%%%%%%%%%%%%%%%%%%%%%%%%%%%%%%%%%%%%%%%%%%%%%%%%%%%%%%%%%%%%%%%%%%%%%%%%%%%%%%%%%%%%%%%%%%%%%%%%%%%%%%%%%%%%%%%%%%%%%%%%%%
%%%%%%%%%%%%%%%%%%%%%%%%%%%%%%%%%%%%%%%%%%%%%%%%%%%%%%%%%%%%%%%%%%%%%%%%%%%%%%%%%%%%%%%%%%%%%%%%%%%%%%%%%%%%%%%%%%%%%%%%%%
%%%%%%%%%%%%%%%%%%%%%%%%%%%%%%%%%%%%%%%%%%%%%%%%%%%%%%%%%%%%%%%%%%%%%%%%%%%%%%%%%%%%%%%%%%%%%%%%%%%%%%%%%%%%%%%%%%%%%%%%%%
%%%%%%%%%%%%%%%%%%%%%%%%%%%%%%%%%%%%%%%%%%%%%%%%%%%%%%%%%%%%%%%%%%%%%%%%%%%%%%%%%%%%%%%%%%%%%%%%%%%%%%%%%%%%%%%%%%%%%%%%%%
\bibliographystyle{amsplain}

%%%%%%%%%%%%%%%%%%%%%%%%%%%%%%%%%%%%%%%%%%%%%%%%%%%%%%%%%%%%%%%%%%%%%%%%%%%%%%%%%%%%%%%%%%%%%%%%%%%%%%%%%%%%%%%%%%%%%%%%%%
%%%%%%%%%%%%%%%%%%%%%%%%%%%%%%%%%%%%%%%%%%%%%%%%%%%%%%%%%%%%%%%%%%%%%%%%%%%%%%%%%%%%%%%%%%%%%%%%%%%%%%%%%%%%%%%%%%%%%%%%%%
%%%%%%%%%%%%%%%%%%%%%%%%%%%%%%%%%%%%%%%%%%%%%%%%%%%%%%%%%%%%%%%%%%%%%%%%%%%%%%%%%%%%%%%%%%%%%%%%%%%%%%%%%%%%%%%%%%%%%%%%%%
%%%%%%%%%%%%%%%%%%%%%%%%%%%%%%%%%%%%%%%%%%%%%%%%%%%%%%%%%%%%%%%%%%%%%%%%%%%%%%%%%%%%%%%%%%%%%%%%%%%%%%%%%%%%%%%%%%%%%%%%%%
%%%%%%%%%%%%%%%%%%%%%%%%%%%%%%%%%%%%%%%%%%%%%%%%%%%%%%%%%%%%%%%%%%%%%%%%%%%%%%%%%%%%%%%%%%%%%%%%%%%%%%%%%%%%%%%%%%%%%%%%%%
\end{document}